\documentclass{amsart}
\usepackage{graphicx}
\usepackage{amssymb}
\newtheorem{thm}{Theorem}[section]
\newtheorem{cor}[thm]{Corollary}
\newtheorem{lem}[thm]{Lemma}

\theoremstyle{definition}

\theoremstyle{question}

\theoremstyle{Conjecture}

\numberwithin{equation}{section}

\begin{document}

\title[groups with given same order type]{Shen's conjecture on groups with given same order type}%
\author{L. Jafari Taghvasani and M. Zarrin}%

\address{Department of Mathematics, University of Kurdistan, P.O. Box: 416, Sanandaj, Iran}%
 \email{L.jafari@sci.uok.ac.ir and Zarrin@ipm.ir}
\begin{abstract}
For any group $G$, we define an equivalence relation $\thicksim$ as below:
\[\forall \ g, h \in G \ \ g\thicksim h \Longleftrightarrow |g|=|h|\]
the set of sizes of equivalence classes with respect to this relation is called the same-order type of $G$ and denote by $\alpha{(G)}$. In this paper, we give a partial answer to a conjecture raised by Shen. In fact, we show that if $G$ is a nilpotent group, then $|\pi(G)|\leq |\alpha{(G)}|$, where $\pi(G)$ is the set of prime divisors of order of $G$.  Also we investigate the groups all of whose
proper subgroups, say $H$ have $|\alpha{(H)}|\leq 2$.\\\\
 {\bf Keywords}.
  Nilpotent groups, Same-order type.\\
{\bf Mathematics Subject Classification (2000)}. 20D60.
\end{abstract}
\maketitle

\section{\textbf{ Introduction and results}}
Let $G$ be a group, define an equivalence relation $\thicksim$ as below:
\[\forall \ g, h \in G \ \ g\thicksim h \Longleftrightarrow |g|=|h|\]
the set of sizes of equivalence classes with respect to this relation is called the same-order type of $G$. For instance, the same-order type of the quaternion group $Q_8$ is $\{1,6\}$. The only groups of type $\{1\}$ are $1$ , $\Bbb{Z}_2$. In \cite{shen}, Shen showed that a group of same-order type $\{1,n\} (\{1,m,n\})$ is nilpotent (solvable, respectively). Furthermore he gave the structure of these groups. In this paper, we give a partial answer to a conjecture raised by Shen in \cite{shen} and we prove that if $G$ is a nilpotent group, then $|\pi(G)|\leq |\alpha{(G)}|$.

Given a class of groups $\mathcal{X}$, we say that a group $G$ is a minimal non-$\mathcal{X}$-group,
or an $\mathcal{X}$-critical group, if $G\not\in X$, but all proper subgroups of $G$ belong to $\mathcal{X}$. It is
clear that detailed knowledge of the structure of minimal non-$\mathcal{X}$-groups can provide
insight into what makes a group belong to $\mathcal{X}$.  For
instance, minimal non-nilpotent groups were analysed by  Schmidt \cite{sch} and proved that such
groups are solvable (see also \cite{zar}). Suppose that $t$ be a positive integer and $\mathcal{Y}_t$ be the class of all groups in which $|\alpha{(G)}|\leq t$. Here, we determine the structure of minimal non-$\mathcal{Y}_2$-group.

Denote by $\phi$ and $S^G_n$ the Euler's function and  the number of elements of order $n$ in a group $G$ respectively. $X_n$ is the set of all elements of order $n$ in a group $G$. We use symbols $\pi_e(G)$ for the set of element orders.\\

\section{\textbf{Shen's conjecture}}

In \cite{shen}, Shen posed a conjecture as follows: \newline
Let $G$ be a group with same-order type $\{1, n_2, \cdots , n_r\}$. Then $|\pi(G)| \leq r$. \newline
Here we give a partial answer to a  this conjecture. Note that by Lemma $3$ of \cite{ab}, we can assume that $G$ is finite. To prove Shen's conjecture we need the following interesting lemmas.

\begin{lem}\label{a}
Suppose that $G$ is a nilpotent group, $m,n\in \pi_e(G)$ and $(m, n)=1$. Then
\begin{align*}
S^G_{mn}=S^G_mS^G_n.
\end{align*}
\end{lem}
\begin{proof}
Let $g\in X_{mn}$. As $(m, n)=1$, so there exist $y, z \in G$, shuch that $o(y)=m$, $o(z)=n$ and $g=yz$. So $g \in X_mX_n$ and $X_{mn}\subseteq X_mX_n$. On the other hand, if $y\in X_m$ and $z\in X_n$, then, as $G$ is nilpotent, we can obtain that $yz=zy$ and so $o(yz)=o(zy)=o(z)o(y)=mn$. It follows that $X_mX_n \subseteq X_{mn}$ and so $X_{mn}=X_mX_n$.
\end{proof}
\begin{cor}
Let $G$ be a nilpotent group, $m\in \pi_e(G)$ and $m=p_1^{h_1}p_2^{h_2} \cdots  p_t^{h_t}$. Then
\begin{align*}
S_m^G=S_{p_1^{h_1}}^G S_{p_2^{h_2}}^G \cdots  S_{p_t^{h_t}}^G.
\end{align*}
\end{cor}
\begin{thm}
Let $G$ be a nilpotent group. Then
\begin{enumerate}
\item If $|\pi (G)| \leq 2$, then
$|\pi (G)|\leq |\alpha{(G)}|.$
\item  If $|\pi (G)| \geq 3$, then
$\pi (G)|\leq |\alpha{(G)}|-1.$
\end{enumerate}
\end{thm}
\begin{proof}
(1).\; If $|\pi (G)|=1$, then $G$ is a $p$-group and obviously $|\pi(G)| \leq |\alpha{(G)}|$. Let $\pi (G)=\{p, q\}$. Since $G$ is nilpotent, $G=P\times Q$, where $|P|=p^n$ and $|Q|=q^m$ are $p$-sylow and $q$-sylow subgroups of $G$, respectively. If $p=2$ and $n=1$, then $G\cong \Bbb{Z}_2 \times Q$. Clearly $\alpha{(G)} =\alpha{(Q)}$. Now if $exp(Q)=q$, then $s_q^Q=q^m-1$. So $|\alpha{(G)}|=|\alpha{(Q)}|=|\pi(G)|=2$. Otherwise if $exp(Q)\neq q$, then there exists $x\in Q$ such that $o(x)=q^2$ and since $S_q\neq S_{q^2}$, so $|\alpha{(Q)}|\geq 3$ and $|\alpha{(G)}|\geq 3> |\pi(G)|$. In other values of $p$ and $n$, in view of Lemma \ref{a}, the conclusion is trivial.

(2).\;By hypothesis since $G$ is nilpotent, so $G=P_1\times \cdots \times P_n$, where $P_i$'s are $p_i$-sylow subgroups of $G$ and $p_1< p_2< \cdots <p_n$. We prove by induction on $n$.
If $n=3$, the $\alpha{(P_1)}\cup \alpha{(P_2)} \cup \alpha{(P_3)}\supseteq \{r,t\}$, for distinct numbers $r$ and $t$, so $\alpha{(G)} \supseteq \{1, r, t, rt\}$, as desired.

Now assume the conclusion is true for $G_{n-1}=P_1\times \cdots \times P_{n-1}$. Let for any $1\leq i\leq n-1$, $\alpha{(P_i)}=\{1, n^i_1, \cdots, n^i_{t_i}\}$ and $S_{{p_i}^{h_i}}$ for $1\leq i\leq n-1$ be the maximum number of the set $\alpha(P_i)$. Now for any $l\in \pi_e(G_{n-1})$, assume that $l=p_1^{\beta _1} \cdots p_r^{\beta _r}$, where $1\leq r \leq n-1$. By the maximality of $S_{{p_i}^{h_i}}$'s , we have
\begin{align*}
S_l=S_{p_1^{\beta _1} \cdots p_r^{\beta _r}}= & S_{{p_1}^{\beta _1}}  \cdots S_{{p_r}^{\beta _r}} \\
                                                                 \leq & S_{{p_1}^{h_1}} \cdots  S_{{p_{n-1}}^{h_{n-1}}} \\
\end{align*}
Besides, $S^G_{p_n} \neq 0$ and since $\phi (p_n) =p_n -1 \mid S_{p_n}$, so $S_{p_n}\neq 1$. Hence we have
\begin{align*}
S_l\leq S_{{p_1}^{h_1}}  \cdots  S_{{p_{n-1}}^{h_{n-1}}} \lneq & S_{{p_1}^{h_1}}  \cdots  S_{{p_{n-1}}^{h_{n-1}}}  S_{p_n} \\
=& S_{{{p_1}^{h_1}} \cdots p_{n-1}^{h_{n-1}}p_n}
\end{align*}
It followes that $S_{{{p_1}^{h_1}} \cdots p_{n-1}^{h_{n-1}}p_n} \in \alpha{(G_{n})} \setminus \alpha{(G_{n-1})}$.
Therefore
\begin{align*}
|\alpha(G_n)|=|\alpha{(G)}|\geq |\alpha{(G_{n-1})}|+1
 \end{align*} and so by induction hypothesis;
\begin{align*}
|\pi (G)|=n=n-1+1< |\alpha{(G_{n-1})}|+1 \leq |\alpha{(G_n)}|=|\alpha{(G)}|.
\end{align*}
and the conclusion is proved.
\end{proof}

\section{On same-order type of subgroups of a group}

In this section, we determine the structure of minimal non-$\mathcal{Y}_2$-group, as follows.
\begin{thm}\label{1}
Let $G$ be minimal non-$\mathcal{Y}_2$-group. Then $G$ is a Frobenius or 2-Frobenius group.
\end{thm}
\begin{proof}
 Let $H$ be a non-trivial proper subgroup of $G$ and $p\in \pi (H)$. Suppose, on the contrary, that $q\in \pi(G)$ and $q\neq p$. Since $p\mid 1+s_p^{H}$ and $q\mid 1+s_q^{H}$, so $s_p^{H}, s_q^{H}\neq \{0, 1\}$, hence $s_p^{H}=s_q^{H}=n_H$. Now as $H$ is nilpotent, according to Lemma \ref{a}, we have $s_{pq}^H=s_p^H s_q^H=n_H^2$, a contradiction. Thus $H$ is a $p$-group. On the other hand, since
\begin{align*}
p\mid 1+s_p^{H}+s_{p^2}^H,
\end{align*}
so $s_{p^2}^H\neq \{1, n_H\}$, since otherwise $p\mid 1$, a contradiction. Hence $s_{p^2}^{H}=0$. It follows that every proper subgroup of $G$ is $p$-group of exponent $p$. If $p, q\in \pi (G)$, then $G$ has no element of order $pq$. If $G$ is nilpotent, then $G$ is a $p$-group of exponent $p$ and it is easy to see that such groups are in $\mathcal{Y}_2$, a contrary.  If $G$ is non-nilpotent, then, as proper subgroup of $G$ has the same-order type $\{1, n\}$, Theorem 2.1 of Shen follows that $G$ is a Schmidt group and so $|\pi (G)|=2$. Now, as $G$ has no element of order $pq$, Theorem A of \cite{xx}, completes the proof.
\end{proof}

\end{document}